\documentclass[reqno,12pt]{amsart}
\usepackage{array}
\usepackage{amsmath}
\usepackage{amsfonts}
\usepackage{amssymb}
\usepackage{enumerate}
\usepackage{amsthm}
\usepackage{amsmath, amscd}
\usepackage{xy}
\xyoption{all}
\usepackage{supertabular}
\usepackage{marvosym}
\usepackage{wasysym}
\usepackage{hyperref}
\usepackage{tikz}
\usetikzlibrary{matrix,arrows,backgrounds}

\usepackage[text={140mm,220mm},centering]{geometry}
\input diagxy
\input xy
\xyoption{all}

\newtheorem{thm}{Theorem}[section]
\newtheorem{cor}[thm]{Corollary}
\newtheorem{lem}[thm]{Lemma}
\newtheorem{prop}[thm]{Proposition}

\theoremstyle{definition}
\newtheorem{defn}[thm]{Definition}

\theoremstyle{remark}
\newtheorem{ex}[thm]{Example}
\newtheorem{rem}[thm]{Remark}

\title[LCS blow-ups]{Locally conformal symplectic blow-ups}%

\author[S. Yang]{Song Yang}
\address{Center for Applied Mathematics, Tianjin University, Tianjin 300072, P. R. China}
\email{syang.math@gmail.com, syangmath@tju.edu.cn}
\thanks{This work is partially supported by NSFC (grant No. 11571242) and the second author is also partially supported by FRFCU (project No. 0903005203329 and No.
106112016CDJXY100005).}

\author[X. Yang]{Xiangdong Yang}
\address{Department of Mathematics, Chongqing University, Chongqing 401331, P. R. China}
\email{xiangdongyang2009@gmail.com, math.yang@cqu.edu.cn}

\author[G. Zhao]{Guosong Zhao}
\address{Department of Mathematics, Sichuan University, Chengdu 610064, P. R. China}
\email{gszhao@scu.edu.cn}

\numberwithin{equation}{section}
\begin{document}
\renewcommand{\labelenumi}{\textit{\alph{enumi})}}

\begin{abstract}
In this paper,  we study the blow-up of a locally conformal symplectic manifold.
We show that there exists a locally conformal symplectic structure
on the blow-up of a locally conformal symplectic manifold
along a compact induced symplectic submanifold.
\end{abstract}

\date{\today}

\keywords{Locally conformal symplectic manifolds, Locally conformal symplectic blow-ups}

\subjclass[2010]{55B35, 55C55}%

\maketitle

%\setcounter{tocdepth}{1}
%\tableofcontents

%====================================================================

\section{Introduction}\label{Intro}

Let $M$ be a smooth manifold.
A symplectic form on $M$ is a $2$-form $\omega \in \Omega^{2}(M)$ satisfying: (1) $d\omega=0$ and (2) $\omega$ is non-degenerate, i.e.
for each $p\in M$ the map
$$
T_{p}M\ni v\longmapsto \omega(v,-)\in T_{p}^{\ast}M
$$
is an isomorphism.
It is of importance to point out that the existence of the symplectic form $\omega$ on $M$ determines pieces of topological data:
the de Rham cohomology of $M$ with even degrees are non-vanishing and the dimension of $M$ is even, denoted by $2n$,
and there exists a homotopy class of reductions of the structural group of the tangent bundle $TM$ to $\textmd{U}(n)\simeq\textmd{Sp}(2n;\mathbb{R})$.
In particular, if $M$ is a complex manifold and $\omega$ is the K\"{a}hler form of a Hermitian metric on $M$ then we say that $(M,\omega)$ is a K\"{a}hler manifold.

In a more general setting, a subclass of almost symplectic manifolds called locally conformal symplectic manifolds (LCS for short) was introduced and studied by Lee \cite{Lee43}, Liebermann\cite{PL54} and Vaisman \cite{IV76,IV85}.
Intuitively, a locally conformal symplectic form is a non-degenerate 2-form $\omega$ which is conformally equivalent to a symplectic form locally.
From a conformal point of view, locally conformal symplectic manifolds can be thought of the closest to symplectic manifolds.
In particular, the locally conformal symplectic manifolds can serve as natural phase spaces of Hamiltonian dynamical systems and from the geometric aspect it appears in the study of contact manifolds and Jacobi manifolds (cf. \cite{BK11,GL84,IV85}).
Likewise, if $M$ is a complex manifold and the locally conformal symplectic form $\omega$ on $M$ is the K\"{a}hler form of a Hermitian metric $h$
then we say that $(M,\omega)$ is a locally conformal K\"{a}hler manifold (LCK for short)(cf. \cite{DO98}).
To make this more precisely, we have the following diagram explaining the relationships between symplectic/K\"{a}hler manifolds and locally conformal symplectic/K\"{a}hler manifolds:
$$
\begin{array}[c]{ccccc}
&\{\textmd{K\"{a}hler\; manifolds}\} &\subset &\{\textmd{LCK\; manifolds}\} \\
&\rotatebox{90}{$\supset$}&&\rotatebox{90}{$\supset$} \\
&\{\textmd{Symplectic\; manifolds}\}  &\subset &\{\textmd{LCS\; manifolds}\} .
\end{array}
$$

It is well known that the blow-up is a very useful operation in symplectic/K\"{a}hler geometry.
In particular, the K\"{a}hler property is preserved under blow-ups.
In the symplectic category, it was McDuff \cite{McDuff84} who first proved that the blow-up of a symplectic manifold along a compact symplectic submanifold also admits a symplectic structure,
moreover, using this symplectic blow-up technique she constructed the first simply-connected, symplectic manifold which is non-K\"{a}hler.
For locally conformal K\"{a}hler manifolds, Tricerri \cite{FT82} and Vuletescu \cite{VV09} proved that the blow-up of a locally conformal K\"{a}hler manifold at a point has a locally conformal K\"{a}hler structure.
In 2013, using the current theory on locally conformal K\"{a}hler manifolds, Ornea-Verbitsky-Vuletescu \cite{OVV13} showed that the blow-up of a locally conformal K\"{a}hler manifold along a submanifold is locally conformal K\"{a}hler if and only if the submanifold is globally conformally equivalent to a K\"{a}hler submanifold.
In the locally conformal symplectic case, Y. Chen and the first named author \cite{CY16} introduced the definition of locally conformal symplectic blow-up of points and proved that the locally conformal symplectic blow-ups of points also admit locally conformally symplectic structures.
Therefore, a natural problem is:
\emph{
What is the locally conformal symplectic blow-up along a submanifold?
}

The purpose of this paper is to study some birational properties of locally conformal symplectic manifolds.
Inspired by the work of McDuff \cite{McDuff84} we give the construction of the locally conformal symplectic blow-up.
In addition, using the same methods of McDuff \cite{McDuff84} and Ornea-Verbitsky-Vuletescu \cite{OVV13}
we prove the following result
\begin{thm}\label{blow-up-thm}
Let $(M,\omega,\theta)$ be a locally conformal symplectic manifold and $Z$ be a compact induced globally conformal symplectic submanifold of $M$,
and let $\pi: \tilde{M}\to M$ be the blow-up of $M$ along $Z$.
Then $\tilde{M}$ also admits a locally conformal symplectic structure $(\tilde{\omega},\tilde{\theta})$ where $\tilde{\theta}=\pi^{*}\theta$.
\end{thm}

This paper is organized as follows.
We devote Section 2 to the preliminary of locally conformal symplectic structures.
In Section 3, we give the construction of locally conformal symplectic blow-up.
This construction is based on the fact that the tangent bundle of a locally conformal symplectic manifold is a symplectic vector bundle.
In Section 4, we give the proof of the main result (Theorem \ref{blow-up-thm}).
Finally, we propose two further problems related to the locally conformal symplectic blow-up.
%==================================================================
\section{Locally conformal symplectic manifolds}\label{lcs manifolds}
In this section we give a rapid review of locally conformal symplectic manifolds.
Assume that $M$ is a smooth manifold of dimension $n\geq 4$.
Intuitively, a \textit{locally conformal symplectic structure} on $M$
is a non-degenerate $2$-form $\omega$ which is locally conformal to a symplectic form.
More precisely, if there exists an open covering $\{U_{\alpha}\}$ of $M$ and
a family of smooth real-valued functions $\{f_{\alpha}: U_{\alpha}\to \mathbb{R}\}$ such that
$\exp{(-f_{\alpha})}(\omega|_{U_{\alpha}})$ is a symplectic form on $U_{\alpha}$, i.e.,
$
d(\exp{(-f_{\alpha})}\omega|_{U_{\alpha}})=0,
$
then we say that $\omega$ is a locally conformal symplectic structure on $M$.

Let
$
\omega_{\alpha}:=\exp{(-f_{\alpha})}(\omega\mid_{U_{\alpha}}),
$
then from definition we have
\begin{eqnarray*}
  0 =d\omega_{\alpha}
    &=&d(\exp{(-f_{\alpha})}\omega)\\
    &=& -\exp{(-f_{\alpha})}(df_{\alpha}\wedge\omega-d\omega)\\
    &=& \exp{(-f_{\alpha})}(d\omega-df_{\alpha}\wedge\omega).
\end{eqnarray*}
on $U_\alpha$.
This implies that
\begin{equation}\label{equ2.1}
d\omega=df_{\alpha}\wedge\omega
\end{equation}
on $U_{\alpha}$.
Likewise, consider the form $\omega_{\beta}:=\exp{(-f_{\beta})}(\omega\mid_{U_{\beta}})$ we get
\begin{equation}\label{equ2.2}
d\omega=df_{\beta}\wedge\omega
\end{equation}
on $U_{\beta}$.
Suppose that $U_{\alpha}\cap U_{\beta}\neq\emptyset$ then from (\ref{equ2.1}) and (\ref{equ2.2}) we obtain
\begin{equation}
(df_{\alpha}-df_{\beta})\wedge\omega=0
\end{equation}
on $U_{\alpha}\cap U_{\beta}$.
Note that $\omega$ is non-degenerate and the wedge product with $\omega$ is injective on $1$-forms,
hence we obtain a globally defined closed $1$-form $\theta:=\{df_{\alpha}, U_{\alpha}\}$ on $M$ which satisfies
\begin{equation}
d\omega=\theta\wedge\omega.
\end{equation}
Equivalently, we have
\begin{defn}[Locally conformal symplectic structure]
Let $M$ be a smooth manifold of dimension $n\geq 4$.
We say that a non-degenerate $2$-form $\omega$ is a \textit{locally conformal symplectic structure},
if there exists a closed $1$-form $\theta$ such that
\begin{equation}\label{LCS defn}
d\omega=\theta\wedge\omega.
\end{equation}
The triple $(M,\omega, \theta)$ is called a locally conformal symplectic manifold.
\end{defn}

Suppose that there exists another $\theta'$ satisfying \eqref{LCS defn},
then $(\theta-\theta')\wedge \omega=0$.
From the Cartan lemma we get $\omega=(\theta-\theta')\wedge \beta$ for some $1$-form $\beta$;
however, this leads to a contradiction with the non-degeneracy of $\omega$.
This implies that $\theta$ is uniquely determined by $\omega$
and we call it the \textit{Lee form} of the LCS manifold.
In particular, if $\theta$ is an exact $1$-form, i.e. $\theta=df$ for some smooth function $f$ on $M$
then $\omega$ is called \textit{globally conformal symplectic} (\textit{GCS} for short)
and it is straightforward to verify that $e^{-f}\omega$ is a symplectic form on $M$.
\begin{ex}
Every LCK manifold is a LCS manifold.
In particular, many well-known non-K\"{a}hler manifolds, such as the Hopf manifolds and the Inoue surfaces and so on,
are LCK manifolds (cf. \cite[Chapter 3]{DO98}).
\end{ex}
\begin{ex}
Let $N$ be a smooth manifold.
Then the cotangent bundle $T^{\ast}N$ is an open symplectic manifold with the symplectic form $d \lambda$,
where $\lambda$ is the canonical $1$-form on $T^{\ast}N$.
If $\theta'$ is a closed $1$-form on $N$,
then $\omega:=d\lambda-\pi^{\ast}\theta'\wedge\lambda$ is a LCS form on $T^{\ast}N$ with the Lee form $\theta=\pi^{\ast}\theta'$,
where $\pi:T^{\ast}N\to N$ is the bundle map.
Moreover, if $\theta'$ is an exact $1$-form
then $\omega=$ is a GCS form.
\end{ex}

\begin{ex}(\cite[Section 5]{BK11})
Let $X$ be a compact contact manifold and let $\phi:X\longrightarrow X$ be a strict contactomorphism,
then there exists a LCS structure on the mapping torus of $X$ with respect to $\phi$.
In particular, we can choose a 3-dimensional contact manifold $X$ such that $X\times S^{1}$ admits no symplectic and complex structures.
This gives rise to an example that is LCS and not LCK.
\end{ex}

Let $\Omega^{*}(M)$ be the space of smooth forms on the LCS manifold $(M,\omega, \theta)$.
We can define the Lichnerowicz differential
\footnote{In the case of LCK manifolds the differential is called the $\theta$-twisted differential and the associated complex (cohomology) is called the Morse-Novikov complex (cohomology).}
by
\begin{eqnarray*}
d_{\theta}:\Omega^{*}(M)&\rightarrow&\Omega^{*+1}(M)\\
\alpha&\mapsto&d\alpha-\theta\wedge\alpha.
\end{eqnarray*}
Furthermore, we have a complex
\begin{equation}\label{5}
\xymatrix@C=0.5cm{
  \cdot\cdot\cdot \ar[r]^{d_{\theta}\,\,\,\,} & \Omega^{k-1}(M) \ar[r]^{\,\,d_{\theta}} & \Omega^{k}(M)\ar[r]^{\,\,d_{\theta}} &  \cdot\cdot\cdot }
\end{equation}
The complex $(\Omega^{*}(M),d_{\theta})$ is called the \emph{Lichnerowicz complex}, and the associated cohomology group
$$
H^{*}_{\theta}(M):=H^{*}(\Omega^{*}(M);d_{\theta})
$$
is called the \emph{Lichnerowicz cohomology}.
This cohomology is a conformal invariant of the locally conformal symplectic manifold,
which is a proper tool in the study of locally conformal symplectic geometry.

%=================================================================

\section{Construction of locally conformal symplectic blow-ups}

In this section, inspired by McDuff's construction of symplectic blow-ups,
we give the construction of blow-up of LCS manifolds along its induced locally conformal symplectic submanifolds
and for more details we refer to McDuff \cite[Section 2 and Section 3]{McDuff84}.

Let $(M,\omega,\theta)$ be a LCS manifold of dimension $2n$.
Then for any $p\in M$ the tangent space $T_{p}M$ is a symplectic vector space with the symplectic bilinear form
$$
\omega_{p}:T_{p}M\times T_{p}M\longrightarrow \mathbb{R}.
$$
This implies that the structural group of the tangent bundle of $M$ is $\textmd{Sp}(2n;\mathbb{R})$;
furthermore, if we fix an orientation on $M$ then the structural group $\textmd{Sp}(2n;\mathbb{R})$ can be reduced to $U(n)$.

\begin{defn}[Induced LCS submanifold]
Let $(M,\omega,\theta)$ be a LCS manifold, and let $i:Z\hookrightarrow M$ be a submsnifold.
We say that $Z$ is an induced locally conformal symplectic submanifold (\textit{ILCS} submanifold for short)
if $i^{\ast}\omega$ is non-degenerate.
\end{defn}
\begin{defn}[Induced GCS submanifold]
We say that $Z$ is an induced globally conformal symplectic submanifold (\textit{IGCS} submanifold for short)
if $Z$ is an ILCS submanifold and the cohomology class $i^{\ast}[\theta]$ vanishes.
\end{defn}

Notice that an IGCS submanifold of a LCS manifold is always a symplectic submanifold.
Now let $(M,\omega,\theta)$ be a LCS manifold, and let $i:Z\hookrightarrow M$ be an ILCS submsnifold
then we have the following lemma.

\begin{lem}
Let $(M,\omega,\theta)$ be a LCS manifold, and let $Z\subset M$ be an ILCS submsnifold.
Then the normal bundle $\mathcal{N}:=\mathcal{N}_{Z/M}$ of $Z$ in $M$ admits a complex vector bundle structure.
\end{lem}

\begin{proof}
Note that the locally conformal symplectic form $\omega$ on $M$ yields a smooth section of the vector bundle $T^{\ast}M\wedge T^{\ast}M$.
The non-degeneration of $\omega$ means that $(TM,\omega)$ is a symplectic vector bundle.
Since $Z$ is an ILCS submsnifold of $M$ the tangent subbuncle $(TZ,\omega|_{Z})$ is a symplectic subbundle of $(TM|_{Z},\omega|_{Z})$.
Define the symplectic complement of $TZ$ in $(TM|_{Z},\omega|_{Z})$ to be the space
$$
{TZ}^{\omega}:=\bigcup_{p\in Z}\{v\in T_{p}M \mid \omega_{p}(v,w)=0, \, \textmd{for}\,\,\textmd{all} \,\,w\in T_{p}Z\}.
$$
On the one hand, we observe that ${TZ}^{\omega}$ is a symplectic vector bundle with symplectic bilinear form $\omega\mid_{Z}$
which can be identified with the normal bundle $\mathcal{N}$.
On the other hand, since we can choose a compatible complex structure on each symplectic vector bundle to make it into a complex vector bundle.
This immediately implies that the normal bundle $\mathcal{N}$ admits a complex vector bundle structure.
\end{proof}

We are now in a position to give the construction of LCS blow-up.
This construction is analogous to the case of symplectic blow-up since the normal bundle $\mathcal{N}$ is a complex vector bundle.
In the rest of this section we follow the lines in \cite{McDuff84} and use the same results and intermediate steps to construct the LCS blow-up.

Let
$
p:\mathbb{P}(\mathcal{N})\to Z,
$
be the projective bundle corresponding to the normal bundle $\mathcal{N}\longrightarrow Z$.
The tautological line bundle over $\mathbb{P}(\mathcal{N})$, denoted by $L$, is defined to be the subbundle of
$\mathbb{P}(\mathcal{N})\times \mathcal{N}$
whose fiber is $\{(l,v)\,|\,v\in l\}$, i.e.,
$$
L:=\{ (l,v) \,|\, (l,v\in l)\in \mathbb{P}(\mathcal{N})\times \mathcal{N} \}.
$$
Then we have the following commutative diagram:
$$
\CD
  L_{0} @> >> L @> q >> \mathbb{P}(\mathcal{N})\\
  @V \pi VV @V \pi VV @V p VV  \\
  \mathcal{N}_{0} @> >> \mathcal{N} @> \varphi >> Z
\endCD
$$
where $q$ and $\pi$ are the projections of $L$ over $\mathbb{P}(\mathcal{N})$ and $\mathcal{N}$ respectively,
and $L_0$ is the complement of the zero section in $L$
and  $\mathcal{N}_{0}$ is the complement of the zero section in $\mathcal{N}$ .

To define the blow-up as a smooth manifold,
we need following notations:
\begin{itemize}
     \item $W$ a closed tubular neighborhood of $Z$ in $M$,
     \item $D$ a subdisc bundle of $\mathcal{N}$ diffeomorphic to $W$,
     \item $\tilde{D}:=\pi^{-1}(D)$ a disc bundle of the complex line bundle $L$.
\end{itemize}
Following McDuff \cite{McDuff84} we have:
\begin{defn}[LCS blow-up]
Let $(M,\omega,\theta)$ be a LCS manifold, and let $Z\subset M$ be an ILCS submsnifold.
The blow-up $\tilde{M}$ of $M$ along $Z$ is the manifold
$$
\tilde{M}:=\overline{M-W}\bigcup_{\partial \tilde{D}} \tilde{D},
$$
where $\partial \tilde{D}$ is identified with $\partial W$ via the diffeomorphism from $D$ to $W$.
\end{defn}

In particular, the map $\pi$ gives rise to an identification of $\tilde{D}-\mathbb{P}(\mathcal{N})$ with $D-Z$,
and thus an identification of $\tilde{M}-\mathbb{P}(\mathcal{N})$ with $M-Z$.
Therefore, on topology we may view
$$
\tilde{M}:=(M-Z)\bigcup \tilde{D}
$$
by equalizing $M-Z$ and $\tilde{D}$
along $W-Z\cong D-Z \cong \tilde{D}-\mathbb{P}(\mathcal{N})$.
There is a natural inclusion $\mathbb{P}(\mathcal{N}) \hookrightarrow \tilde{M}$,
and we call the projective bundle $\mathbb{P}(\mathcal{N})$ the \textit{exceptional divisor} of the blow-up $\pi: \tilde{M}\to M$ along $Z$.

\begin{rem}\label{rem3.5}
Note that the construction of LCS blow-up depends on the complex vector bundle structure of the normal bundle $\mathcal{N}$ and the tubular neighbourhoods.
Therefore, this construction is not canonical;
however, we can choose the compact tubular neighborhood $W$ of $Z$ in $M$ sufficiently small.
\end{rem}

\section{Proof of the main result}

In this section we give the proof of Theorem \ref{blow-up-thm}.
We use the same method as \cite[Section 3]{McDuff84} and for the reader's convenience we first recall this argument.

Let $(U,\omega)$ be a symplectic manifold and let $i:Z\hookrightarrow U$ be a compact symplectic submanifold of codimension $2k$.
Consider the normal bundle $\pi:\mathcal{N}\longrightarrow Z$ of $Z$ in $U$.
Since $\mathcal{N}$ has a complex vector bundle structure the fiber $\mathcal{N}_{x}$ for each $x\in Z$ admits a canonical exact symplectic form.
From another aspect, in the horizonal direction the zero section of $\mathcal{N}$, still write as $Z$, is a symplectic manifold with the symplectic
form $\omega_{Z}:=i^{*}\omega$.
Choose a local trivialization of $\mathcal{N}$,
i.e. an open covering $\{U_{i}\}$ of $Z$ such that $\mathcal{N}\mid_{U_{i}}\cong U_{i}\times \mathbb{C}^{k}$.
For each $i$ there exists a 1-form $\alpha_{i}$ on $\mathcal{N}\mid_{U_{i}}$ satisfying:
\begin{itemize}
  \item [(1)] for any $x\in U_{i}$ the restriction of $d\alpha_{i}$ on the fiber $\mathcal{N}_{x}$ is the canonical symplectic form;
  \item [(2)] $\alpha_{i}$ is zero on $U_{i}$.
\end{itemize}
Let $\{f_{i}\}$ be a partition of unity subordinate to the open covering $\{U_{i}\}$ then we may construct a closed 2-form on $\mathcal{N}$,
denoted by
$$
\rho=\pi^{*}\omega+\sum_{i}d(f_{i}\alpha_{i}).
$$
In particular, $\rho$ restrict to the canonical symplectic form on each fiber and to $\omega_{Z}$ on $Z$.

According to \cite[Lemma 3.2]{McDuff84}, there exists a closed $2$-form $\alpha$ on $\mathbb{P}(\mathcal{N})$
such that $\alpha$ restricts to the K\"{a}hler form of the canonical Fubini-Study metric on each fibre of
$p:\mathbb{P}(\mathcal{N})\longrightarrow Z$
and the pull-back of $\alpha$ under $q^{\ast}$ is an exact form on $L_{0}$.
Since $q^{\ast}\alpha$ is exact on  $L_{0}$ we have $q^{\ast}\alpha=d \beta$ for some $1$-form $\beta$ on $L_{0}$.
Let $\tilde{U}:=\overline{U-W}\bigcup_{\partial \tilde{D}} \tilde{D}$ be the symplectic blow-up of $U$ along $Z$.
We can choose a constant $\varepsilon=\varepsilon(\rho,\alpha)>0$
and a smooth function $b$ on $\tilde{D}$ which equals 1 near $\mathbb{P}(\mathcal{N})$ and 0 $\partial \tilde{D}$.
Define a closed 2-form $\tilde{\rho}$ on $\tilde{D}$ by setting
$$
\tilde{\rho}
:=\left\{
    \begin{array}{ll}
       \pi^{\ast}\rho+\varepsilon q^{\ast} \alpha  & \textmd{on}\;  \mathbb{P}(\mathcal{N}),\\
      \pi^{\ast}\rho+\varepsilon d(b \beta)  & \textmd{on}\; \tilde{D}- \mathbb{P}(\mathcal{N}).
    \end{array}
  \right.
$$
We may choose suitable $\varepsilon$ such that $\tilde{\rho}$ is non-degenerated on $\tilde{V}:=\pi^{-1}(V)$, where $V$ is a neighborhood of $Z$.
Hence the $2$-form
$$
\tilde{\omega}
:=\left\{
    \begin{array}{ll}
      \omega & \textmd{on}\; U-W \\
       \tilde{\rho} & \textmd{on}\; \tilde{D}
    \end{array}
  \right.
$$
is non-degenerated and closed, i.e. it is a symplectic form on $\tilde{U}$.
More precisely, we have the following key result in the symplectic blow-up.
\begin{prop}(\cite[Proposition 3.7]{McDuff84})\label{McDuff-prop}
Suppose $(U,\omega)$ be a symplectic manifold,
and $i:Z\hookrightarrow U$ be a compact symplectic submanifold (i.e., $i^{\ast}\omega$ is a symplectic form).
Let $\pi:\tilde{U}\to U$ be the blow-up of $U$ along $Z$.
Then there exists a symplectic form $\tilde{\omega}$ on $\tilde{U}$
such that
$$
\tilde{\omega}|_{\tilde{U}-\pi^{-1}(V)}=\pi^{\ast}\omega,
$$
for some neighborhood $V$ of $Z$.
\end{prop}

Now we are in a position to prove Theorem \ref{blow-up-thm}.
Assume that $(M,\omega,\theta)$ is a LCS manifold.
Let $Z\subset M$ be an IGCS submanifold,
thus the restriction of the Lee form $\theta|_{Z}$ is exact.
By a conformal rescaling of the LCS form $\omega$ we may assume that $\theta|_{Z}=0$.
In fact, if $\theta|_{Z}=df$, we denote $\omega':=\exp(-f)\omega$,
then $d\omega'=\exp(-f)(-df\wedge\omega+\theta\wedge \omega)=(\theta-df)\wedge \omega=0$.
Actually, an IGCS submanifold is a symplectic submanifold.
In the rest of this section we prove the following
\begin{thm}[Theorem \ref{blow-up-thm}]
Assume that $(M,\omega,\theta)$ is a LCS manifold and $i:Z\hookrightarrow M$ is a compact induced symplectic submanifold.
Let $\pi: \tilde{M}\to M$ be the LCS blow-up of $M$ along $Z$.
Then $\tilde{M}$ also admits a LCS structure $\tilde{\omega}$ with the Lee form $\tilde{\theta}=\pi^{*}\theta$.
\end{thm}

\begin{proof}
%%%
By assumpation, the pull back $\theta|_{Z}:=i^{\ast}\theta$ is zero.
Let $U$ be a neighborhood of $Z$ such that the inclusion $i:Z\hookrightarrow U$ induces an isomorphism on the first de Rham cohomology groups
$$\xymatrix@C=0.5cm{
   & i^{*}:H^{1}_{dR}(U) \ar[r]^{\,\,\,\cong} & H^{1}_{dR}(Z). }
$$
Via a conformal change of the LCS form $\omega$, we may assume that $\theta|_{U}=0$.
It follows that $\omega|_U$ is a symplectic form on $U$.
In particular, since $\theta|_{U}=0$  the intersection of the support of $\theta$ with $U$ is empty.
Choose a sufficiently small closed tubular neighborhood $W$ of $Z$ in $M$ such that $W\subset U$.
Let $\pi:\tilde{M}\to M$ be the LCS blow-up of $M$ along $Z$ with respect to the tubular neighborhood $W$
and a subdisc bundle $D$ of $\mathcal{N}$ that is diffeomorphic to $W$.
Let $\tilde{U}:=\pi^{-1}(U)$ then $\pi:\tilde{U}\longrightarrow U$ is the symplectic blow-up of $U$ along $Z$ with respect to the neighborhoods $W$ and $D$, i.e.
$$
\tilde{U}:=\overline{U-W}\bigcup_{\partial \tilde{D}} \tilde{D}.
$$
From Proposition \ref{McDuff-prop}, there exists a symplectic form $\tilde{\omega}_{U}$ on $\tilde{U}$,
which equals to $\pi^{\ast}\omega$ outside of $\pi^{-1}(V)$ for a neighborhood $V$ of $Z$ in $U$.
Observe that $\pi^{-1}(Z)=\mathbb{P}(\mathcal{N})$ and $\pi$ gives rise to an identification between $\tilde{M}-\mathbb{P}(\mathcal{N})$ and $M-Z$;
therefore, we obtain a non-degenerate 2-form $\tilde{\omega}$ on $\tilde{M}$ given by
$$
\tilde{\omega}
:=\left\{
    \begin{array}{ll}
      \pi^{*}\omega & \textmd{on}\; \tilde{M}-\tilde{U} \\
       \tilde{\omega}_{U} & \textmd{on}\; \tilde{U}.
    \end{array}
  \right.
$$
It remains to verify that $\tilde{\omega}$ is a LCS form with Lee form $\tilde{\theta}=\pi^{*}\theta$.
It is straightforward since we have $\tilde{\theta}\mid_{\tilde{U}}=0$ and $\tilde{\omega}=\pi^{*}\omega$ outside of $\tilde{U}$.
This completes the proof.
\end{proof}

Under the LCS blow-ups we also have a blow-up formula of the Lichnerowicz cohomology as following.
\begin{cor}(\cite[Theorem 1.1]{YZ15})
Let $(M,\omega,\theta)$ be a compact LCS manifold of dimension $2n$.
Assume that $Z\subset M$ is a compact IGCS submanifold of codimension $2r$.
Then we have
$$H^{k}_{\theta}(M)\oplus\biggl(\bigoplus^{r-2}_{i=0} H^{k-2i-2}_{dR}(Z)\biggr)\cong H^{k}_{\tilde{\theta}}(\tilde{M}),$$
where $\pi:\tilde{M}\longrightarrow M$ is the LCS blow-up of $M$ along $Z$.
\end{cor}

%=====================================================
\section{Concluding remark}

In \cite[Corollary 2.11]{OVV13}, using the current theory on complex manifolds, Ornea-Verbitsky-Vuletescu proved that if the blow-up of a compact LCK manifold along a  compact submanifold admits a LCK structure then the submanifold must be an IGCK submanifold.
Similarly, for LCS manifolds we have the following problem:
\emph{
If the blow-up of a compact LCS manifold along a compact ILCS submanifold admits a LCS structure,
is it true that this submanifold is IGCS?
}
It is noteworthy that for LCS manifolds we can not use the current theory since the almost complex structures on LCS manifolds are not integrable necessarily.

The existence of K\"{a}hler metrics on a compact complex manifold implies many topological properties.
These properties enable people to construct many examples of non-K\"{a}hler and symplectic manifolds.
In the case of LCK geometry, it is not easy to exclude whether a manifold admits a LCK metric for the lack of topological obstructions.
Comparing with symplectic/K\"{a}hler geometries, Ornea-Verbitsky \cite{OV11} proposed an open problem:
\emph{
Construct a compact LCS manifold which admits no LCK metrics.
}

In 2011, Bande-Kotschick \cite{BK11} constructed a 4-dimensional product manifold $M\times S^{1}$ which is LCS and not LCK.
Later, in 2014 Bazzoni-Marrero \cite{BM14} constructed a symplectic, and hence LCS, nilmanifold $N$ which is not the product of a compact 3-manifold and a circle (see also \cite[Corollary 3.6]{BM14a}).
In particular, they proved that $N$ admits no complex structures.
This implies that $N$ is LCS and not LCK.
In fact, using the LCS blow-up technique at points of LCS manifolds having no LCK structures,
we can obtain more LCS manifolds without any LCK structures (cf. \cite[Corollary 2.4]{CY16}).
Furthermore, a natural problem is:

\emph{
How to construct examples of LCS manifolds which are not symplectic and LCK?
Moreover, can we construct higher dimensional LCS manifolds without any complex structures?
}

%====================================================================\

\end{document}